\documentclass[reqno]{amsart}
\usepackage[T1]{fontenc}
\usepackage[french]{babel}
\usepackage{amssymb, amsmath}

\usepackage[pdftex]{hyperref}

\numberwithin{equation}{section}

\theoremstyle{plain}
\newtheorem{theorem}{Th\'eorème}[section]
\newtheorem{cor}[theorem]{Corollaire}
\newtheorem{prop}[theorem]{Proposition}
\newtheorem{lemma}[theorem]{Lemme}
\newtheorem{hyp}{Hypothèse}

\newtheorem*{teo}{Théorème}

\theoremstyle{definition}

\newtheorem{remark}[theorem]{Remarque}

\newtheorem{definition}[theorem]{D\'efinition}

\newtheorem*{rmque}{Remarque}

\newcommand{\nc}{\newcommand}

\newcommand{\Z}{\mathbb{Z}}
\newcommand{\Q}{\mathbb{Q}}
\newcommand{\N}{\mathbb{N}}
\newcommand{\C}{\mathbb{C}}
\newcommand{\M}{\mathcal{M}}
\newcommand{\Max}{\mathfrak{m}}

\newcommand{\F}{\mathbb{F}}
\nc\e{\epsilon}
\nc\LL{\mathcal L}
\nc\bG{\mathbb{G}}
\newcommand{\Aut}{\operatorname{Aut}}
\newcommand{\Autf}{\operatorname{Aut}_f}
\newcommand{\Stab}{\operatorname{Stab}}
\newcommand{\St}{\operatorname{St}}
\newcommand{\si}{\sigma}

\nc{\dd}{\delta}
\nc{\THT}{\Theta}
\nc{\tht}{\theta}
\nc{\sscl}[1]{\langle #1 \rangle}
\nc{\alg}[1]{#1^\textrm{alg}}
\nc{\restr}[1]{\!\!\upharpoonright_{#1}}

\newcommand{\inv}{ ^{-1}}
\newcommand{\tp}{\operatorname{tp}}

\newcommand{\cl}{\operatorname{cl}}

\nc\Frob{\mathrm{Frob}}

\def\Ind#1#2{#1\setbox0=\hbox{$#1x$}\kern\wd0\hbox to 0pt{\hss$#1\mid$\hss}
\lower.9\ht0\hbox to 0pt{\hss$#1\smile$\hss}\kern\wd0}
\def\Notind#1#2{#1\setbox0=\hbox{$#1x$}\kern\wd0\hbox to
0pt{\mathchardef\nn="0236\hss$#1\nn$\kern1.4\wd0\hss}\hbox to
0pt{\hss$#1\mid$\hss}\lower.9\ht0
\hbox to 0pt{\hss$#1\smile$\hss}\kern\wd0}
\def\ind{\mathop{\mathpalette\Ind{}}}
\def\nind{\mathop{\mathpalette\Notind{}}}
\def\indi#1{\mathop{\ \ \hbox to 0pt{\hss$\mid^{\hbox to
0pt{$\scriptstyle#1$\hss}}$\hss}
\lower4pt\hbox to 0pt{\hss$\smile$\hss}\ \ }}
\def\nindi#1{\mathop{\ \ \hbox to 0pt{\hss$\!\not{\mid}^{\hbox to
0pt{$\scriptstyle\,#1$\hss}}$\hss}
\lower4pt\hbox to 0pt{\hss$\smile$\hss}\ \ }}

\def\indild{\mathop{\ \ \hbox to 0pt{\hss$\mid^{\hbox to
0pt{$\scriptstyle\mathrm{ld}$\hss}}$\hss}
\lower4pt\hbox to 0pt{\hss$\smile$\hss}\ \ }}

\begin{document}


\title{Sur les automorphismes bornés de corps munis d'opérateurs}
\date{\today}

\author{Thomas Blossier, Charlotte Hardouin et Amador Martin-Pizarro}
\address{Universit\'e de Lyon; CNRS; Universit\'e Lyon 1; Institut Camille
Jordan UMR5208, 43 boulevard du 11
novembre 1918, F--69622 Villeurbanne Cedex, France 
\newline
\indent
Institut de Mathématiques de Toulouse UMR5219, Université
Paul Sabatier, 118 route de Narbonne, F--31062 Toulouse Cedex 9,
France.
\newline
\indent Universit\'e de Lyon; CNRS; Universit\'e Lyon 1; Institut
Camille Jordan UMR5208, 43 boulevard du 11
novembre 1918, F--69622 Villeurbanne Cedex, France.}
\email{blossier@math.univ-lyon1.fr}
\email{charlotte.hardouin@math.univ-toulouse.fr}
\email{pizarro@math.univ-lyon1.fr}
\thanks{Cette collaboration a débuté lors du programme thématique 
\emph{Model Theory, Arithmetic Geometry and Number Theory} au printemps 
2014 au MSRI, que nous remercions.  Le travail du  second auteur a  reçu le  soutien 
du   projet ANR-11-LABX-0040-CIMI dans le cadre 
du programme ANR-11-IDEX-0002-02 ainsi que du projet ANR-10-JCJC 0105. 
Les premier et troisième auteurs ont 
conduit cette recherche gr\^ace au soutien du projet ValCoMo 
ANR-13-BS01-0006, sans subvention ni du projet ANR-10-LABX-0070, 
ni du projet ANR-11-IDEX-0007.}

\begin{abstract} 
Nous donnons une preuve alternative, valable en toute caractéristique,
d'un ancien résultat de Lascar caractérisant les automorphismes bornés
d'un corps algébriquement clos. Cette méthode se généralise au cas des 
automorphismes bornés de certains corps munis d'opérateurs.
\end{abstract}

\maketitle

\section*{English Summary} 
We give an alternative proof, valid in all characteristics, of a
result of Lascar characterising the bounded automorphisms of an
algebraically closed field. We generalise this method to various fields 
equipped with operators.  

\section*{Introduction}

Le groupe des automorphismes d'une structure dénombrable,
muni de la topologie de la convergence ponctuelle, est un groupe  
\emph{polonais}, un groupe topologique homéomorphe à un espace métrique 
séparable complet.
Le groupe des automorphismes $\Aut(\M)$ d'une structure  $\M$ contient 
un sous-groupe distingué, le sous-groupe $\Autf(\M)$ des 
automorphismes forts. Un automorphisme est \emph{fort} s'il fixe  les 
classes de   toute 
relation d'équivalence définissable sans 
paramètres avec un nombre fini de classes. 
Par exemple, pour le corps $\C$ des nombres complexes, il s'avère que les
automorphismes forts coïncident avec les automorphismes fixant $\alg \Q$. 
Ainsi, le quotient $\Aut(\C)/\Autf(\C)$ est le groupe de Galois absolu de 
$\Q$. Plus généralement, le quotient $\Aut(\M)/\Autf(\M)$, qui dépend 
uniquement de la théorie $T$ de $\M$, est un groupe profini, nommé le 
\emph{groupe de Galois} de $T$.

Le groupe de Galois d'une théorie est loin d'être compris en général. En 
revanche, pour les théories fortement minimales (qui incluent 
celles des corps algébriquement clos), le sous-groupe $\Autf(\M)$ est 
simple modulo les automorphismes forts bornés \cite{dL92}. Rappelons qu'un 
élément est  \emph{algébrique} sur une sous-partie $A$ s'il satisfait une 
formule à paramètres sur $A$ ayant un nombre fini des réalisations. Ceci 
correspond à la notion usuelle d'éléments algébriques sur $A\subset \C$. 
Pour une théorie fortement minimale, l'opérateur cl\^oture algébrique 
satisfait le principe de l'échange et induit une dimension, qui correspond 
au degré de transcendance pour les corps algébriquement clos. 

Un automorphisme $\tau$ est \emph{borné} s'il existe un  ensemble 
fini $A$ tel que pour tout élément $b$, 
l'image $\tau(b)$ est algébrique sur $A\cup\{b\}$. (Cette définition est 
équivalente à la définition originale de Lascar, voir \cite[Preuve du 
Théorème 15]{dL92}.)
 
 La simplicité du groupe $\Autf(\M)$ modulo les automorphismes forts bornés
 a été généralisée \cite{EGT} à toute structure $\M$ munie d'une dimension 
à valeurs entières avec une relation d'indépendance stationnaire compatible 
avec cette dimension, en remplaçant \emph{algébrique} par \emph{de 
dimension relative} $0$ dans la définition des automorphismes  bornés.

Pour un pur corps algébriquement clos $K$ de
caractéristique $0$,  Lascar montre \cite[Théorème 15]{dL92}
que l'unique automorphisme borné est l'identité. Ainsi $\Autf(K)$
est simple. Il mentionne que Ziegler a décrit les automorphismes bornés 
d'un pur corps algébriquement clos en  toute caractéristique \cite{mZ91}~: 
les seuls automorphismes bornés d'un corps algébriquement clos en 
caractéristique positive sont les puissances entières du Frobenius.

Konnerth montre la trivialité des automorphismes bornés d'un corps 
différentiellement clos de caractéristique $0$, en prenant  pour notion de 
dimension  le  degré de transcendance différentiel \cite[Proposition 
2.9]{rK02}.

D'autres exemples classiques de corps munis d'un opérateur sont les
corps aux différences génériques : modèles
existentiellement clos dans la classe des corps dans le langage des anneaux
enrichi par un symbole dénotant un automorphisme. Ces corps sont
obtenus en prenant des ultraproduits de puissances de Frobenius
sur la clôture algébrique de $\alg{\F_p}$ \cite{Hr04}. Un grand nombre des 
propriétés modèle-théoriques valables dans le cas des purs corps ou des 
corps différentiels en caractéristique $0$ s'étendent à ce contexte. Dans 
\cite{MS14}, Moosa et Scanlon proposent une approche qui englobe les corps 
différentiellement clos et les corps aux différences
génériques en caractéristique $0$~: les \emph{corps munis
d'opérateurs libres}. Cependant, leur méthode ne s'applique ni en 
caractéristique positive, ni lorsque les opérateurs commutent, ce qui 
exclut, en particulier, deux exemples connus en caractéristique nulle: les
corps différentiels avec $n$ dérivations qui commutent, étudiés par 
McGrail \cite{McG00}, et les corps différentiels aux différences, traités 
par Hrushovski ainsi que par Bustamante-Medina \cite{rB11}.

Un élément d'un corps $K$ muni d'opérateurs est \emph{générique} si
ses mots en les opérateurs ne satisfont aucune relation algébrique,
mis à part celles imposées par la théorie. Contrairement aux cas des
corps différentiels ou aux différences, un élément non-générique d'un
corps muni de plusieurs opérateurs n'est pas toujours
fini-dimensionnel~: il peut engendrer une sous-structure de degré de
transcendance infini. Il n'existe à priori pas de notion naturelle de
dimension pour définir quand un automorphisme est borné. En revanche,
à partir d'un (tout) type générique, on peut définir un opérateur
cl\^oture adéquat~: la cl\^oture d'une partie $D$ de $K$ est
l'ensemble des éléments \emph{co-étrangers} sur $D$ aux génériques 
(\emph{cf.} définition \ref{D:cl}).
Dans le cas des purs corps algébriquement clos, cette clôture
correspond à la clôture algébrique~; dans le cas des corps
différentiellement clos ou aux différences génériques, la clôture  de
$D$ est l'ensemble des  éléments fini-dimensionnels au-dessus de $D$. 

 Un automorphisme $\tau$ est alors borné s'il existe un ensemble fini $A$ 
 tel que pour tout générique $b$ sur $A$, l'image $\tau(b)$ appartient à la clôture 
de $A \cup\{b\}$. 

Dans cet article, nous donnons une preuve uniforme (cf. Théorème \ref{T:noborne}), qui s'inspire  
de celle de Ziegler pour les purs corps algébriquement clos \cite{mZ91}, 
 permettant de caractériser les automorphismes bornés des corps munis d'opérateurs suivants~:
\begin{teo}

Pour les corps munis d'opérateurs suivants~:

\begin{itemize}
\item les corps algébriquement clos $(K,\mathrm{Id})$ en toute 
caractéristique avec automorphisme associé l'identité~;
\item les corps différentiellement clos $(K,\dd_1,\ldots,\dd_n)$  en 
caractéristique nulle avec $n$ dérivations qui commutent avec 
automorphisme associé l'identité~;
\item les corps aux différences génériques $(K,\sigma)$ en toute 
caractéristique avec automorphisme associé $\sigma$~;
\item les corps séparablement clos de degré d'imperfection fini avec une 
$p$-base nommée et automorphisme associé  l'identité~;
\item les corps différentiels aux différences $(K,\dd,\si)$ en 
caractéristique nulle avec automorphisme associé $\si$, qui commute 
avec $\dd$~;
\item les corps $(K,F_1,\ldots,F_n)$ munis d'opérateurs libres en 
caractéristique nulle, avec automorphismes associés 
$\si_0,\ldots, \si_t$~;
\end{itemize}

\noindent tout automorphisme borné est un produit des automorphismes 
associés et de leurs inverses, ainsi que du Frobenius et son 
inverse, si le corps est parfait de caractéristique positive. 
\end{teo}

En particulier, on retrouve le résultat de Ziegler~: le
groupe des automorphismes forts d'un corps algébriquement clos
en caractéristique positive  est simple modulo le sous-groupe
cyclique engendré par le Frobenius. 

\begin{rmque}
Notons que les automorphismes associés sont des automorphismes 
corpiques mais pas nécessairement des automorphismes de la 
structure~: par exemple, pour un corps muni de deux automorphismes 
libres distincts, chacun de ces automorphismes est un automorphisme associé.  
\end{rmque}

Nous remercions le rapporteur anonyme pour ses suggestions et remarques qui nous ont permis d'améliorer la présentation de ce travail. Le troisième auteur tient à s'excuser auprès de F. O. Wagner pour  
avoir mis du temps à comprendre l'intérêt de la clôture analysable. En s'inspirant 
de cet article, Wagner a par ailleurs démontré \cite{Wa15} que tout 
automorphisme borné d'un corps ayant une théorie simple est 
définissable. 

\section{Corps munis d'opérateurs}

Dans ce travail, nous considérons des corps munis d'opérateurs additifs.  Les premiers exemples sont les corps  algébriquement clos, les corps différentiellement 
 clos, ainsi que les corps algébriquement clos aux différences. Moosa et Scanlon \cite{MS14} ont 
développé un formalisme pour traiter simultanément ces trois cas. 
Pour la présentation des résultats, nous allons supposer que les opérateurs sont en nombre fini. Pour les corps séparablement clos de degré d'imperfection fini \cite{cW79, fD88}, voir la remarque \ref{R:SCF}.

\begin{definition}\label{D:corpsOp}
Un \emph{corps muni d'opérateurs} sur un sous-corps de base 
$\F\subset K$ est
une structure  
$$(K,0,1,+,-,\cdot,\{\lambda\}_{\lambda\in\F},F_1,\ldots,F_n)$$
 telle que~:
\begin{enumerate}
 \item les opérateurs $F_1,\ldots,F_n$ sont $\F$-linéaires et vérifient
pour tous $x$ et $y$ dans $K$,
\[ F_k(xy)=\sum\limits_{0\leq i,j\leq n} a_{i,j}^k  F_i(x)F_j(y),\]
\noindent pour certaines constantes $\{a_{i,j}^k\}_{0\leq i,j,k\leq n}$ dans $\F$ 
(avec $F_0$ l'identité)~;
\item le 
$\F$-espace vectoriel $\F \e_0 \oplus \ldots \oplus \F \e_n$ est une 
$\F$-algèbre commutative, avec
\[\e_i  \e_j = \sum\limits_{0\leq k\leq n} a_{i,j}^k \e_k.\]
\end{enumerate} 
 
\end{definition} 
\noindent Une telle structure est bi-interprétable avec 
$(K,0,1,+,-,\cdot,\{\lambda\}_{\lambda\in\F}, D(K), \varphi)$, où
$D(K)$ est une $K$-algèbre égale au $K$-espace vectoriel 
$K \e_0 \oplus \ldots \oplus K \e_n$  de  dimension $n+1$, avec un 
morphisme de $\F$-algèbres $\varphi:K\to D(K)$ tel que la projection de 
$D(K)$ sur la première 
coordonnée composée avec $\varphi$ soit l'identité. 

\noindent Pour cela, il suffit de 
poser \[\varphi(x)=\sum\limits_{0\leq k\leq n} F_k(x) \e_k.\] Cette
approche est celle de  \cite{MS14}.

La $\F$-algèbre $D(\F)=\F \e_0 \oplus \ldots \oplus \F \e_n$ étant de 
dimension finie, elle est isomorphe à un produit de $\F$-algèbres locales 
$B_0(\F), \ldots, B_t(\F)$ \cite[Theorem 8.7]{AM69}. Les corps résiduels de 
ces algèbres locales sont des extensions finies de $\F$, et donc égaux à 
$\F$ dès que $\F$ est algébriquement clos.   Moosa et Scanlon donnent un 
exemple sur $\Q$ avec deux opérateurs \cite[Example 4.2]{MS14} de corps 
résiduel $\Q(\sqrt 2)$. En revanche, ce phénomène n'existe pas pour 
un seul opérateur non trivial.

\begin{prop}\label{P:calculsPourUnOperateur}
Soit $(K,0,1,+,-,\cdot,\{\lambda\}_{\lambda\in\F}, F)$ un corps muni d'un
seul opérateur $F$. Alors, il existe des  constantes $a$, $b$ et $c$ dans 
$\F$ vérifiant $b^2-b = ac$ telles que
$$F(xy) = axy+b(xF(y)+yF(x))+cF(x)F(y)$$ pour tous $x$ et $y$ dans $K$. De 
plus, l'algèbre $D(\F)$ est ou bien locale de corps résiduel $\F$, ou 
bien isomorphe à $\F^2$. 

Enfin, la structure est soit  bi-interprétable avec celle de pur corps 
au-dessus de $\F$, soit bi-interprétable avec $K$ muni d'une dérivation ou 
d'un endomorphisme non-trivial.
\end{prop}
\begin{proof}
Si $F=\lambda\mathrm{Id}$, pour un certain $\lambda$ dans $\F$, alors 
$b=c=0$ et $a=\lambda$ conviennent. L'algèbre $D(\F)$ est isomorphe 
à 
$\F^2$ et la structure est bi-interprétable avec celle de pur corps 
au-dessus de $\F$. Supposons maintenant que $F$ soit $\F$-linéairement 
indépendant de l'identité et que la $\F$-algèbre $D(\F)=\F \e_0 \oplus  \F 
\e_1 $ soit
donnée par les relations suivantes : 
$$
\left. \begin{array}{cc}
        \e_0^2 &= \e_0 + a\e_1 \\
         \e_0\e_1 &= \alpha \e_0 + b \e_1  \\
        \e_1^2 &=\beta \e_0 + c \e_1
    \end{array} \right\}
$$
où les constantes $a$, $b$, $c$, $\alpha$ et $\beta$ appartient à
$\F$. En particulier, puisque $\varphi$ est un homomorphisme
d'anneaux, on obtient que 
\begin{multline*}
\varphi(xy)= xy\e_0  + F(xy)\e_1= \Big(xy+
\alpha(xF(y)+yF(x)) +\beta F(x)F(y)\Big)\e_0 + \\ 
\Big(axy+b(xF(y)+yF(x))
+c F(x)F(y)\Big)\e_1.\end{multline*}

\noindent Puisque la projection de $\varphi$ sur la première 
coordonnée est
l'identité sur $K$, on en déduit que $\alpha(xF(y)+yF(x)) +\beta
F(x)F(y)=0$ pour tous $x$ et $y$ dans $K$. En posant $x=y$, on obtient pour
tout $x$ dans $K$,
$$F(x) (2 \alpha x + \beta F(x)) =0.$$
Vérifions que $\alpha = \beta =0$. Si $\alpha =0$, alors $\beta =0$, car 
$F$ n'est pas l'opérateur trivial nul. Supposons $\alpha\neq 0$. Dans ce 
cas, l'opérateur $F$ ne s'annule qu'en $0$~: soient $x_0$ et $y_0$ tels que
$F(x_0)\neq 0$ et $F(y_0) = 0$. Alors $F(x_0+y_0) = F(x_0) \neq 0$ et donc 
$2 \alpha (x_0+y_0) = -\beta F(x_0+y_0)  = -\beta F(x_0) = 2 \alpha x_0$. 
On conclut que $y_0 =0$ et $F$  est  colinéaire à l'identité.

Comme $\alpha = \beta =0$, il suit que  l'idéal engendré par
$\e_1$ est un idéal maximal de $D(\F)$. 

L'égalité  $\e_0^2\e_1=\e_0(\e_0\e_1)$ nous permet de déduire que
$b^2=b+ac$. De plus, l'opérateur $F$ satisfait :
$$F(xy) = axy+b(xF(y)+yF(x))+cF(x)F(y).$$

Si $c=0$, alors $b\neq 0$, car $F$ est linéairement indépendant de
l'identité. Puisque $(\e_0+\lambda \e_1)\e_1=b\e_1$ pour tout $\lambda$
dans $\F$, le seul idéal maximal de $D(\F)$ est l'idéal engendré par
$\e_1$. Comme $1=\varphi(1)\equiv \e_0 \mod (\e_1)$, on conclut que
$D(\F)$ est locale avec corps résiduel $\F$. De plus, comme  
$b^2=b+ac$, alors $b=1$ et l'application $\dd = F + a \mathrm{Id}$ est une 
dérivation.

Si $c\neq 0$, alors on pose $\e'_1=c\inv \e_1$ et
$\e'_0=\e_0-b\e'_1$. On obtient ainsi une base orthonormale de
$D(\F)$, qui est donc isomorphe à $\F^2$. L'application 
$\si = cF+b\mathrm{Id}$ est alors un morphisme de corps qui ne peut être 
trivial.
\end{proof}

\begin{remark}\label{R:residue_nder}
Les opérateurs sur les corps considérés par Moosa et Scanlon sont libres, 
dans le sens qu'ils ne satisfont aucune relation \cite[Remark 3.8]{MS14}. 
Ici, nous n'imposons pas cette liberté pour les opérateurs. Ainsi, les 
corps différentiels avec $n$ dérivations qui commutent et les corps 
différentiels aux différences sont également des corps munis d'opérateurs 
selon la 
définition \ref{D:corpsOp}. Soit $(K,\dd_1,\ldots,\dd_n)$ un corps muni de 
$n$ 
dérivations qui commutent. Notons par $\F$ le corps premier de $K$ et soit
$D(\F)$ l'algèbre correspondante engendrée par 
$1=\e_0,\e_1,\ldots,\e_n$ avec les relations suivantes~:

$$
\left. \begin{array}{ccl}
        \e_i^2 &= 0  &\text{ si } 0<i\leq n\\
         \e_i\e_j &= 0 &\text{ si } i\neq j \text{ et } i,j>0.
    \end{array} \right\}
$$
Alors $D(\F)$ est locale avec idéal maximal $\Max=(\e_1,\ldots,\e_n)$ et 
corps résiduel $\F$.

Si $(K,\dd,\si)$ est un corps muni d'une 
dérivation et d'un endomorphisme qui commutent \cite{rB11}, alors, 
au-dessus du corps premier $\F$, l'algèbre $D(\F)$ engendrée par 
$\{\e_0,\e_1,\e_2\}$ satisfait les relations suivantes~:

$$
\left. \begin{array}{rl}
         1 &= \e_0+\e_2 \\
         \e_i^2 &= \e_i  \text{\hskip4mm  si } i\neq 1\\
         \e_1^2 &=0 \\
         \e_0\e_1 &= \e_1 \\
         \e_0\e_2 &= 0 \\
         \e_1\e_2 &= 0.
    \end{array} \right\}
$$
Les seuls idéaux maximaux de $D(\F)$ sont 
$\Max_0=(\e_1,\e_2)$ et $\Max_1=(\e_0, \e_1)$. Notons que 
$(\Max_0\Max_1)^2=(\e_1)^2=0$. Chaque algèbre $B_i(\F)=D(\F)/\Max_i^2$ 
est locale d'idéal maximal l'image de $\Max_i$ modulo $\Max_i^2$ et corps 
résiduel $\F$. L'algèbre $D(\F)$ est isomorphe au produit $B_0(\F)\times 
B_1(\F)$, par le théorème des restes chinois. 
\end{remark}

\begin{hyp}\label{H:Residue} (\emph{cf.} Assumption 4.1 \cite{MS14}) 
Chaque algèbre locale associée à $D(\F)$ a pour corps résiduel 
$\F$.
\end{hyp}

En tensorisant chaque algèbre locale par $K$, si $\theta_i$, 
respectivement $\rho_i$, dénote  la projection de $D(K)$ sur $B_i(K)$, 
respectivement la
projection résiduelle de $B_i(K)$ sur $K$, on obtient un endomorphisme
 $\sigma_i =  \rho_i\circ \theta_i \circ \varphi$ de $K$. Notons que
l'identité est l'un de ces endomorphismes, que l'on suppose être
$\sigma_0$. Ces endomorphismes sont combinaisons $\F$-linéaires des
opérateurs et ils ne dépendent que de la structure 
$(K,0,1,+,-,\cdot,\{\lambda\}_{\lambda\in\F}, D(K), \varphi)$. En 
particulier, si l'on considère une transformation $\F$-linéaire inversible 
des opérateurs, on obtient une structure bi-interprétable ayant les mêmes 
endomorphismes associés. 

\noindent Dans le cas d'un seul opérateur $F$ avec constantes $a$,
$b$ et $c$, comme dans la proposition \ref{P:calculsPourUnOperateur}, les 
endomorphismes obtenus sont l'identité et  $\si=b\mathrm{Id}+cF$ si 
$c\neq 0$. Lorsque $c=0$, l'identité est le seul endomorphisme associé.  
Il en est de même pour les corps avec $n$ dérivations qui commutent. Si
$(K,\dd,\si)$ est comme dans la remarque \ref{R:residue_nder}, alors les 
endomorphismes associés sont l'identité et  $\si$.

\begin{hyp}\label{H:Endo}
Par la suite, les endomorphismes $\sigma_1,\ldots,\sigma_t$ du corps 
 muni d'opérateurs $(K,0,1,+,-,\cdot,F_1,\ldots,F_n)$ sont des 
automorphismes de corps. Ceci est le cas pour les modèles 
existentiellement clos dans la classe des corps en caractéristique nulle 
munis d'opérateurs libres \cite[Lemma 4.11]{MS14}, ainsi que pour les corps 
aux différences 
génériques, en toute caractéristique \cite{ChHr99} ou pour les corps 
différentiels aux différences \cite{rB11}. 
\end{hyp} 

\begin{prop}\label{P:Triang_Charlotte}
Étant donné un corps muni d'opérateurs 
$$(K,0,1,+,-,\cdot,\{\lambda\}_{\lambda\in\F},F_1,\ldots,F_n)$$ 
satisfaisant l'hypothèse (\ref{H:Residue}), il existe une transformation 
$\F$-linéaire inversible qui rend les opérateurs \emph{triangulaires}~:

\[F_j(xy)=\sigma_{i_j} (x) F_j(y) + \sum\limits_{l<j} R_{j,l}(x) F_l(y),\]
où chaque $R_{j,l}(x)$ est un polynôme à coefficients sur $\F$ en les
symboles $\{F_r(x)\}_{0\leq r\leq j}$.
\end{prop}

\begin{proof}
Par une première transformation linéaire inversible, on peut supposer que 
la base linéaire $(\e_0, \ldots, \e_n)$ du $\F$-espace vectoriel $D(\F)=\F 
\e_0 
\oplus \ldots \oplus \F \e_n$ est l'union de bases linéaires des algèbres 
locales $B_i(\F)$ associées. 

Il reste alors à traiter le cas de chaque algèbre locale, que l'on dénote 
$(B(\F), \Max)$. Par l'hypothèse \ref{H:Residue} et le théorème principal
de Wedderburn-Mal'cev \cite[p. 209]{Pie82}, l'algèbre $B(\F)$ peut 
s'écrire, en tant que $\F$-espace vectoriel, comme une somme directe de 
$\F 1_{B(\F)}$ et $\Max$. Prenons une base $\F$-linéaire
$\{\e'_1,\ldots,\e'_{r}\}$ de $\Max$.  Notons $\sigma$ l'endomorphisme de 
$K$ correspondant et 
$\theta$ la projection de $D(K)$ sur $B(K)$. Alors pour tout $x$ dans $K$,
$$\theta \circ \varphi (x) = \sigma(x) + F'_1(x) \e'_1 +\ldots + F'_r(x) 
\e'_r $$ où les $F'_j$ (y compris $\sigma = F'_0$) sont les opérateurs 
obtenus par ce changement de base.

La multiplication par chaque $\e'_j$  est une application nilpotente de 
$B(\F)$, car $\Max^s=0$ pour $s\gg 0$. De plus, les applications 
$\e'_j$ 
commutent entre elles. Il existe ainsi une base $\{\eta_1,\ldots,\eta_{r}\}$
de $\Max$ telle que toutes les applications multiplication par  $\e'_j$ 
pour $j \leq r$ sont triangulaires inférieures dans cette base avec unique 
valeur propre $0$. Puisque chaque $\eta_k$ est  combinaison linéaire des 
$\e'_j$, les multiplications par $\eta_k$ sont également triangulaires 
inférieures dans la base $\{\eta_1,\ldots,\eta_{r}\}$. De plus, comme ces 
applications commutent, 
on obtient des constantes $\{b_{k,l}(j)\}_{j,k,l \leq r} \in \F$
telles que :
\begin{equation}\label{eq:sum} \tag{\dag}
\eta_k \eta_l =\sum\limits_{j>\max{k,l}} b_{k,l}(j) \,
\eta_j.
\end{equation} 

Ainsi, après une nouvelle transformation linéaire inversible, le morphisme
d'anneaux $\theta \circ \varphi$ s'écrit sous la forme

$$\theta \circ \varphi (x) = \sigma(x)  + F''_1(x) \eta_1 +\ldots +
F''_r(x) \eta_r \,.$$ 

\noindent L'identité $\theta \circ \varphi(xy)=\theta \circ \varphi(x) 
\cdot \theta \circ \varphi(y)$ et (\ref{eq:sum}) donnent 

$$\begin{array}{rcl} F''_j(xy)&= &\sigma(x) F''_j(y) 
+  F''_j(x) \sigma(y) +\sum\limits_{0<k,l < j} b_{k,l}(j) F''_k(x)F''_l(y)\\
&=& \sigma(x) F''_j(y) + \sum\limits_{l<j} R'_{j,l}(x) F''_l(y),
\end{array}$$
où $F''_0(y) = \sigma(y)$, $R'_{j,0}(x) = F''_j(x)$ et
$R'_{j,l}(x) = \sum\limits_{0<k<j}  b_{k,l}(j) F''_k(x)$ pour $0<l<j$.

En recollant les bases obtenues pour chacune des algèbres locales, 
on en déduit le résultat voulu.
\end{proof}

Nous considérons par la suite la famille $\THT$  des mots dans les 
opérateurs $F_1,\ldots,F_n$, $\sigma_1\inv,\ldots,\sigma_t\inv$, 
ainsi que   $\Frob$ (où $\Frob$ désignera l'endomorphisme de Frobenius 
en caractéristique positive et  l'identité 
en caractéristique nulle) munie de
l'ordre lexicographique à partir des relations suivantes :

$$ \si_i\inv < \si_j\inv < F_i < F_j < \Frob \text{ pour } i<j.$$

\noindent \`A chaque combinaison  $K$-linéaire $S(x)$ de mots en $x$, 
l'on associe 
son \emph{degré}, le plus grand mot qui apparaît avec coefficient non-nul 
dans $S$. Ce coefficient sera appelé le \emph{coefficient dominant} de $S$.

\begin{cor}\label{C:Trig_Charlotte}
Supposons que les opérateurs $F_1,\ldots, F_n$ sont triangulaires et,
suivant les notations de la proposition \ref{P:Triang_Charlotte}, 
posons 
$\si_{F_j} =\si_{i_j}$ et $\si_{\si_j\inv}=\si_j\inv$, ainsi que $\si_{\Frob} = \Frob$. 
Ainsi, chaque mot $\tht$ de la famille $\THT$ détermine un produit $\si_{\tht}$ en les endomorphismes 
 correspondants aux lettres dans $\tht$, qui sera alors une composition d'une puissance du Frobenius avec des 
puissances entières des automorphismes associés au corps muni d'opérateurs. 

Si $S(x)$ est une combinaison $K$-linéaire de mots en $x$ de degré $\tht$
et coefficient dominant $\lambda_{\tht}$, alors pour tout $g$ dans $K$  on a

 \[S(gx)=\lambda_{\tht}\si_{\tht} (g) \tht(x) + R(x),\]
où $R(x)$ est une combinaison $K$-linéaire de mots en $x$ de degré 
strictement inférieur à $\tht$.
\end{cor}

\begin{remark}\label{R:Trig_SCF}
Notons que nous avons besoin d'un nombre fini d'opérateurs et de 
l'hypothèse $(\ref{H:Residue})$ pour pouvoir appliquer la proposition 
\ref{P:Triang_Charlotte} et rendre les opérateurs triangulaires. 
Pour les corps séparablement clos de degré  d'imperfection fini, chacune des 
dérivations de Hasse-Schmidt itératives  \cite{mZ03} est une suite infinie  mais déjà 
sous forme triangulaire avec automorphisme associé l'identité, donc ils satisfont
 l'hypothèse $(\ref{H:Endo})$. Le corollaire précédent s'applique.
\end{remark}

\section{Une promenade modèle-théorique}

Pour cette partie, nous nous plaçons à l'intérieur d'un corps 
$K$ muni d'opérateurs satisfaisant l'hypothèse $(\ref{H:Endo})$. 

\noindent Nous supposons de plus par la suite~:

\begin{hyp}\label{H:AlgebriquementClosPlusFrob}
Le corps $K$ est séparablement clos (éventuellement 
algébriquement clos), suffisamment saturé 
et homogène dans le langage $\LL$ qui 
étend celui des anneaux par des constantes pour  les éléments du sous-corps
 $\F$, par des symboles $F_1,\ldots,F_n, \sigma_1\inv,\ldots,\sigma_t\inv$ 
 pour les opérateurs et les automorphismes associés.
\end{hyp}

Étant donnée une partie $A$, nous notons $\sscl{A}$ le corps de 
fractions de la sous-structure engendrée par $A$ à l'intérieur de $K$. Si 
$k\subset K$ est un sous-corps, on dénote  $\alg k$ sa clôture 
algébrique au sens corpique.

Nous allons isoler des propriétés modèles-théoriques pour 
étudier les automorphismes bornés d'un tel corps. Deux uples $a$ et $b$ 
ont même type sur une partie  $D$, noté $a \equiv_D b$, s'il existe un
automorphisme fixant $D$ qui envoie $a$ sur $b$, pourvu que $K$ soit 
$|D|^+$-fortement homogène. Pour les
 corps munis d'opérateurs libres en caractéristique 
nulle \cite[Proposition 5.5 et Proposition 5.6]{MS14}, ainsi que pour les 
corps différentiels avec $n$ dérivations qui commutent \cite[Lemma 
3.15]{McG00} et pour les corps différentiels aux différences 
\cite[Proposition 1.1 et Theorem 1.3]{rB11}, les 
types dans la théorie de $K$ sont déterminés par le type aux différences de 
la $\LL$-structure engendrée par une réalisation, ce qui motive l'hypothèse
suivante.

\begin{hyp}\label{H:types}. La clôture algébrique 
(au sens modèle-théo\-ri\-que) d'un sous-ensemble $A$ coïncide avec $\alg
{\sscl{A}} \cap K$.

En outre, deux uples $a$ et $b$ ont même type sur $k=\sscl k$ si et 
seulement s'il existe un  $\LL$-isomorphisme entre $\alg {\sscl{k(a)}}\cap K$ et 
$\alg {\sscl{k(b)}}\cap K$ qui envoie $a$ sur $b$ fixant $k$.
\end{hyp}

Étant données des sous-parties $A$, $B$ et $C$ de $K$, l'on dit que $A$ 
\emph{est indépendant de} $B$ sur $C$, dénoté $A\ind_C B$, si les 
extensions $\alg{\sscl {A\cup C}}$ et $\alg{\sscl {B\cup C} }$ sont
linéairement disjointes sur $\alg{\sscl C}$. Notons que cette relation est 
symétrique, transitive, de caractère fini et invariante par 
automorphismes de $K$. De plus, elle a caractère local : pour tout uple
fini $a$ et toute partie $B$, il existe  $B_0\subset B$ de taille bornée par $|\LL|$, avec 
$a\ind_{B_0} B$.

\begin{remark}\label{R:pairwise}
Puisque chaque mot en $a\cdot b$ ou en $a+b$ s'exprime
comme une combinaison de mots en $a$ et $b$, si les
éléments $a$ et $b$ sont indépendants sur $C$, alors $a$ et $a\cdot
b$, ainsi que $a$ et $a+b$, le sont aussi.
\end{remark}

\begin{hyp}\label{H:Simple}
L'indépendance définie ci-dessus satisfait les conditions suivantes~:
\begin{enumerate}
\item Pour tout uple fini $a$ et toutes sous-parties $C\subset B$, il 
existe $a'\equiv_C a$ avec $a'\ind_C B$~; 
\item Étant donné un sous-corps $k=\alg{\sscl k}\cap K$ et des 
uples $a$ et $b$ avec $a\equiv_k b$, pour toutes parties $C$ et $D$ 
contenant  $k$ telles que $C\ind_k D$, si $a\ind_k C$ et $b\ind_k D$, alors
il existe $e\ind_k C\cup D$ tel que $e\equiv_{C} a$ et $e\equiv_{D} b$.
\end{enumerate}

En particulier, le théorème de Kim-Pillay \cite{KP97} donne que la 
$\LL$-théorie $T$ de $K$ est simple et que l'indépendance correspond à la 
non-déviation, définie par Shelah. 
\end{hyp}
Les exemples considérés satisfont l'hypothèse $(\ref{H:Simple})$~: voir 
\cite[Section 4.3]{McG00} pour les corps différentiels avec $n$ dérivations
qui commutent, ainsi que \cite[Theorem 5.9]{MS14} pour les corps munis 
d'opérateurs libres en caractéristique nulle et \cite[Theorem 1.3]{rB11} 
pour les corps différentiels aux différences. 

\noindent 
Nous renvoyons le lecteur à \cite{Wa00} pour une
introduction aux théories simples. 

\begin{definition}\label{D:SMorley}
La suite $\{a_i\}_{i<\omega}$ est \emph{indépendante} sur l'ensemble
$B$ si, pour tout $m<\omega$, on a 
\[a_{m+1}\ind_B a_1,\ldots,a_m.\]
La suite $\{a_i\}_{i<\omega}$ est \emph{indiscernable} sur l'ensemble
$B$ si pour tous $i_1<\ldots<i_m$, on a $a_1,\ldots,a_m \equiv_B
a_{i_1},\ldots,a_{i_m}$. 

\noindent
Étant donné un type $p$ à paramètres sur $B$, une  \emph{suite de
Morley} $\{a_i\}_{i<\omega}$ de $p$ est une suite indépendante et
indiscernable sur $B$ de réalisations de $p$. 
\end{definition}

Dans une théorie simple, tout type admet des suites de
Morley. Afin de faciliter la présentation, 
sans introduire ni des hyperimaginaires ni des imaginaires, nous supposons 
par la suite  que la théorie $T$ satisfait la condition suivante~:

\begin{hyp}\label{H:Forking}

Si $B=\alg{\sscl B}\cap K$, et $\{a_i\}_{i<\omega}$ est une suite de Morley
 du type $\tp(a/B)$, alors
\[a\ind_X B,\]
avec $X=B\cap\alg{\sscl{\{a_i\}_{i<\omega}}}$.
\end{hyp}

Cette propriété est toujours vérifiée si $T$ est supersimple avec  
élimination des imaginaires  \cite{BPW01}, ce qui est le cas des corps
algébriquement clos, des corps différentiellement clos de caractéristique 
$0$ avec $n$ dérivations qui commutent \cite[Corollary 3.3.2]{McG00}, des 
corps aux  différences génériques en toute caractéristique et des corps  
différentiels aux différences \cite[Proposition 3.36]{rB07}. Pour les corps 
munis d'opérateurs libres en  caractéristique $0$, cette propriété suit de 
\cite[Theorem 5.12 et Claim  6.17]{MS14}.

\begin{definition}\label{D:typedefgp}

Un groupe $G$ est \emph{type-définissable} dans $T$ sur un ensemble $A$ de
paramètres si son domaine est  l'intersection d'ensembles
définissables sur $A$ muni d'une loi de groupe relativement 
définissable sur $A$ (Si la théorie $T$ est stable ou supersimple, 
alors $G$ est l'intersection de groupes
définissables sur $A$ \cite[Remark 5.5.1 et Theorem 5.5.4]{Wa00}). Il est 
\emph{connexe} sur $A$ s'il n'a pas de sous-groupes propres 
type-définissables sur $A$ d'indice borné (par rapport à la 
saturation de $K$).

Un élément $g$ de $G$ est générique  sur $A$ si, pour tout $h$ dans $G$ 
avec $g\ind_A h$, alors $$h\cdot g \ind  A\cup\{h\}.$$ 
\end{definition}

Toute extension non-déviante d'un générique l'est aussi. 
Tout élément de $G$ s'écrit comme le produit de deux génériques. En outre, 
le produit de deux génériques indépendants sur $A$ l'est aussi et est 
indépendant sur $A$ de chaque facteur. Si $C$ est un translaté à droite 
d'un sous-groupe $H\leq G$, le tout type-définissable sur $A$, un élément 
$c \in C$ est générique sur $A$, si pour un $x \in C$, le translaté $x 
\cdot c \inv $ est générique dans $H$ au-dessus de $A \cup \{x\}$.

%

\begin{definition}\label{D:Stab}
Dans un groupe  ambiant $G$ type-définissable sur un ensemble 
$A=\alg{\sscl{A}}\cap K$,  le \emph{stabilisateur (à gauche)} $\Stab(g/A)$ d'un élément $g$ 
sur $A$ est un sous-groupe type-définissable sur $A$, qui est 
défini par $\Stab(g/A)=\St(g/A)\cdot \St(g/A)$, avec
$$\St(g/A) = \{h\in
  G\ :\ \exists\,x\models\tp(g/A)\ (\,hx\models\tp(g/A) \ \land
  x\ind_A h\,)\}.$$ 
\end{definition}
\noindent
Tout générique du sous-groupe $\Stab(g/A)$ est contenu dans $\St(g/A)$. Un élément $g$  de $G$ est 
générique sur $A$ si et seulement 
si $\Stab(g/A)$ est d'indice borné. 
 
Le résultat suivant a été démontré par Ziegler \cite[Theorem~1]{mZ90} dans
le cas stable abélien, et généralisé à tout groupe type-définissable dans 
une théorie simple \cite[Lemme 1.2 et Remarque 1.3]{BPW14}. 

\begin{lemma}\label{L:stab} Étant donnés deux éléments 
$a$ et $b$ d'un groupe  $G$ type-définissable dans $T$ tels que $a$, 
$b$ et $a\cdot b$ sont deux-à-deux indépendants au-dessus d'un ensemble
des paramètres $A=\alg{\sscl A}\cap K$, alors $a$ et $a\cdot b$ ont même
stabilisateur sur $A$, commensurable avec un conjugué de 
$\Stab(b/A)$. Ces
stabilisateurs sont connexes sur $A$. De plus,
chaque élément est générique dans le translaté à droite par cet élément de
son stabilisateur respectif. Ce translaté est aussi définissable sur $A$.
\end{lemma}
\noindent Si $G$ est abélien, alors les trois éléments ont même
stabilisateur sur $A$. 

Dans un groupe algébrique sur un pur corps algébriquement clos, les stabilisateurs obtenus dans le lemme précédent seront des sous-groupes algébriques. 

\begin{cor}\label{C:sousgroupesAdditifs}
Un sous-groupe $H$ de $\bG_a^k(K)$ type-définissable et connexe sur 
$D=\alg{\sscl D} \cap K$  d'indice non-borné est contenu dans un sous-groupe
d'indice non-borné de la forme 
$$\{(x_1,\ldots,x_k) \in \bG_a^k(K) \,|\, S(x_1,\ldots,x_k) =0\},$$
où $S$ est une combinaison $D$-linéaire non triviale en les mots en $x_i$.
\end{cor}

Même si les opérateurs ne sont pas nécessairement libres, la relation 
$S=0$ du corollaire précédent ne peut être imposée par la théorie.

\noindent Rappelons que le Frobenius apparaît dans les mots en $x_i$ en caractéristique positive. 

\begin{proof}
Soit $a$ et $b$ deux génériques de $H$ indépendants sur $D$. Ainsi $a$, 
$b$ et $a+b$ sont deux-à-deux indépendants sur
$A$~:  les structures 
$\sscl{D(a)}$, $\sscl{D(b)}$ et $\sscl{D(a+b)}$ sont deux-à-deux
algébriquement indépendantes sur $D$. Puisque $H$ n'est pas d'indice borné,
l'uple $a$ n'est pas un générique de $\bG_a^k(K)$. Par les propriétés sur les génériques et l'indépendance, il existe une suite finie de mots $\bar\tht$, commune aux uples $a$, $b$ et $a+b$, qui témoigne 
une relation d'algébricité au-dessus de  $\alg{D}$ non-imposée par la théorie $T$. Dans le pur corps algébriquement
clos $\alg{K}$, le lemme \ref{L:stab} appliqué aux uples $\bar\tht(a)$, $\bar\tht(b)$ et $\bar\tht(a+b)$ au-dessus de  $\alg{D}$  entraîne que  leur stabilisateur $H_1$ est un 
sous-groupe algébrique défini sur $\alg{D}$ de la puissance correspondante  du groupe additif
$\bG_a$. Le translaté $H_1 + \bar\tht(a)$ est également défini sur $\alg{D}$. Le sous-groupe $H_1$ est propre car $\bar\tht(a)$ n'est pas algébriquement indépendant, et donné par des polynômes additifs. Ainsi, il existe une combinaison $\alg{D}$-linéaire $S'$ non triviale en les mots en $x_i$ et un élément $d' \in \alg{D} $ tel que $S'(a)=d'$. Si $K$ n'est pas algébriquement clos, il suffit de composer par une puissance suffisante du Frobenius pour obtenir une combinaison $D$-linéaire $S$ et $d \in D$ tel que $S(a) =d$. 

Si $a'\equiv_D a$ est indépendant de $a$ sur $D$, alors $a-a'$ est un 
générique de $H$ sur $D$ satisfaisant l'équation $S(x_1,\ldots,x_k) =0$, ce 
qui donne le résultat par connexité de $H$. 
\end{proof}

\begin{remark}\label{R:Ssgps_add_SCF}
Notons que si le corps muni d'opérateurs n'était pas séparablement clos, par exemple, les corps 
pseudo-finis, nous aurions besoin d'une description (à indice fini prés) des sous-groupes additifs 
définissables. 
\end{remark}

\begin{definition}\label{D:cl}
Pour un type générique fixé $p$ de $K$ sur $\emptyset$, on définit la
$p$-clôture $\cl_p(D)$ d'un ensemble $D$ comme la collection
des éléments $x$ de $K$ \emph{co-étrangers} sur $D$ à $p$, c'est-à-dire 
tels que, pour tout $D_1 \supset D$ et toute réalisation $a$ de $p$, 
générique sur $D_1$, on ait $a \ind_{D_1} x$. 
\end{definition}

Cette définition correspond à la définition usuelle \cite[Definition 
3.5.1]{Wa00}, par \cite[Remark 5.1.19]{Wa00}.

\begin{remark}\label{R:cl}
La cardinalité de $\cl_p(D)$ peut être comparable à celle de $K$, même 
si $D$ est fini. Par exemple, dans le cas des corps différentiels avec $n$ 
dérivations qui commutent, comme tout générique a rang $\omega^n$ 
\cite[Corollary 5.2.8]{McG00}, cette clôture contient tous les éléments 
non-génériques et, en particulier, les constantes. 

Cette clôture ne dépend pas du générique $p$ choisi. En effet~: soient 
$q$ un générique et $x$ un élément de $\cl_q(D)$. Considérons $D_1 \supset 
D$  et une réalisation $a$ de $p$, générique sur $D_1$. Soit $b$ une 
réalisation de $q$ générique sur $D_1 \cup \{a\}$. Alors $b$ est encore 
générique sur $D_1 \cup \{b\inv\cdot a\}$, donc
$$b \ind_{D_1 \cup \{b\inv\cdot a\}} x,$$
ce qui entraîne 
$$a \ind_{D_1 \cup \{b\inv\cdot a\}} x \text{ et  par transitivité }
a \ind_{D_1} x.$$
En particulier, si $x$ est dans la clôture de $D$, il ne peut être
générique sur $D$. La réciproque est vraie pour les purs corps
algébriquement clos, les corps aux différences génériques et les corps
différentiels aux différences \cite[Corollary 2.9]{rB11}. Ainsi, cette
clôture correspond à la clôture algébrique pour les purs corps
algébriquement clos et, dans le cas des corps différentiellement clos
ou aux différences génériques, à la collection des éléments
fini-dimensionnels au-dessus de $D$. \cite[Definition 6.1]{MS14} 

Par la suite, la  clôture de $D$ sera notée simplement $\cl(D)$.
\end{remark}

Dans la remarque précédente, on a noté que la clôture de $D$
correspondait à la collection de non-génériques sur $D$ si la théorie
a rang de Lascar $\omega^\alpha$, pour un certain ordinal $\alpha$.
 Nous remercions Z. Chatzidakis de nous avoir fourni les pistes nécessaires
 pour observer que ce n'est pas le cas pour les corps munis d'opérateurs 
libres.

\begin{remark}
Soit un corps algébriquement clos en caractéristique nulle
existentiellement clos suffisamment saturé muni de deux automorphismes
$\sigma_1$ et $\sigma_2$ qui ne commutent pas. 

Par \cite[Theorem 4.6]{MS14}, il existe un élément $b$ tel que les
relations algébriques entre les mots en $b$ sont uniquement celles induites
par l'équation $\sigma_1(b)=b$. On vérifie facilement que
$a=b+\sigma_2(b)$ est générique mais $a\nind b$, ce qui donne un exemple 
d'élément non-générique qui n'appartient pas à la clôture de $\F$. 

En revanche,  la clôture $\cl(D)$ n'est pas réduite à $\alg{\sscl
D}$~: considérons  $D=\alg{\sscl
D}$ et $x$ un élément.  Modulo les identifications $\sigma_i\inv \sigma_i 
= \sigma_i \sigma_i \inv = \mathrm{Id}$, notons  $\THT_r$ l'ensemble des 
mots en $\sigma_1$, $\sigma_2$, $\sigma_1\inv$, $\sigma_2\inv$ de longueur 
au plus $r$ (Le mot vide de longueur $0$ correspond à l'identité). 

Posons $$f_x(r)=\mathrm{deg.tr}
(\THT_r(x)/D),$$
où $\THT_r(x)$ est l'uple constitué des éléments $\theta(x)$ pour $\theta 
\in \THT_r$. 

On a $|\THT_1 \setminus \THT_0| =4$ et $|\THT_{r+2} \setminus \THT_{r+1}| 
= 3 \times |\THT_{r+1} \setminus \THT_{r}|$. Ainsi, pour  $a$ générique sur 
$D$, 
$$\mathrm{deg.tr}(\THT_{r+1}(a)/D(\THT_{r}(a))) = f_a(r+1) - f_a(r) = 
4\times 3^r.$$  Si $a\nind_D x$, alors il existe des entiers $r_0$ et 
$s_0$ tels que $$\mathrm{deg.tr}
(\THT_{r_0+1}(a)/D(\THT_{r_0}(a),\THT_{s_0}(x))) \leq 4\times 3^{r_0}
-1.$$
Puisque pour tous entiers $r$ et $s$, 
$$\mathrm{deg.tr}
(\THT_{r+2}(a)/D(\THT_{r+1}(a), \THT_{s+1}(x))) \leq 3
\mathrm{deg.tr}
(\THT_{r+1}(a)/D(\THT_{r}(a), \THT_{s}(x))), $$
on obtient pour tout $m$ l'inégalité
$$\begin{array}{rcl}
4\times 3^{r_0+m} - f_x(s_0+m) &\leq &\mathrm{deg.tr}
(\THT_{r_0+m+1}(a)/D(\THT_{r_0+m}(a),\THT_{s_0+m}(x)))\\[1mm]
&\leq& 3^m (4\times 3^{r_0} -1),
\end{array}
$$ donc 
$$f_x(s_0+m)) \geq 3^m.$$
Ainsi la clôture $\cl(D)$ contient tout élément $x$ vérifiant 
$f_x(r)=o(3^r)$. Or, par compacité, il existe  un élément $c$ tel que les 
seules relations algébriques entre les mots en $c$ sont  celles induites 
par les équations $\sigma_1(\sigma_2^l(c))=\sigma_2^l(c)$ pour $l \in \Z$. 
Cet élément 
$c$ n'est pas algébrique mais $f_c(r) = 2r+1$, donc $c \in \cl(D)$.

On peut vérifier de manière analogue que pour tout corps muni d'au moins 
deux opérateurs libres, la clôture ne contient pas tous les non-génériques 
et n'est pas réduite à la clôture algébrique. 
\end{remark}

Le caractère fini de l'indépendance et la transitivité donnent
le résultat suivant.

\begin{cor}\label{C:indi}
Pour tout générique $a$ sur $D$, $$ a\ind_D \cl(D).$$
\end{cor}

Comme aucun générique sur un ensemble ne peut être dans la clôture 
algébrique (au sens modèle-théorique) de ce 
dernier, on a également la propriété suivante.

\begin{remark}\label{R:indi} 
Si $b$ est un générique sur $D$ qui appartient  à $\cl(D,a)$, alors 
$ b\nind_D a$.
\end{remark}

La remarque suivante est une adaptation directe de \cite[Lemma 
3.5.5]{Wa00}.

\begin{remark}\label{R:etranger} 
Si $c$ appartient à $\cl(B)$ et $c\ind_A B$ avec
$A\subset B$, alors $c$ appartient à $\cl(A)$. 
\end{remark}
\begin{proof}
Soient $D\supset A$ et un générique $b$ sur $D$. Pour montrer que $b\ind_D 
c$, on peut supposer, par l'hypothèse $(\ref{H:Simple}) (1)$, que 
$Db\ind_{Ac} B$. Puisque $c\ind_A B$, alors 
$$ Dbc\ind_A B,$$
donc $b\ind_D D\cup B$ et $b$ est un générique sur $D\cup B$. Puisque 
$c\in\cl(B)$, on a 

$$ b\ind_{D\cup B} c,$$
ce qui entraîne $b\ind_D c$, par transitivité. 
\end{proof}

\begin{definition}\label{D:borne}
Un automorphisme $\tau$ de la $\LL$-structure $K$ est \emph{borné} 
s'il existe un ensemble $D$ fini tel que, pour tout  générique  $a$ sur 
$D$, l'élément $\tau(a)$ appartient à $\cl(D,a)$.
\end{definition}

Par la remarque \ref{R:cl}, on retrouve ainsi la définition originale de 
Lascar pour un pur corps algébriquement clos, ainsi que celle de Konnerth 
pour les corps différentiellement clos. 

\begin{lemma}\label{L:independanceSurCloture}
Soit $\tau$ un automorphisme borné sur un ensemble fini de paramètres $D$,
et $a$ et $b$ deux éléments génériques et indépendants sur $D$. Si
l'on pose $$D_0 = \cl(D) \cap \alg{\sscl {D(a, b,
\tau(a),\tau(b))}},$$ alors les paires 
$(a,\tau(a))$, $(b,\tau(b))$ et $(ab,\tau(ab))$ (respectivement, les
paires $(a,\tau(a))$, $(b,\tau(b))$ et
$(a+b,\tau(a+b))$) sont
deux-à-deux
indépendantes sur $D_0$. 
\end{lemma}
Par le corollaire \ref{C:indi}, l'élément $a$ reste indépendant sur 
$D_0$. De même pour $\tau(a)$, s'il était générique sur $D$. 
\begin{proof} On peut supposer par la suite que $D=\alg{\sscl D}\cap K$. Le 
corollaire $\ref{C:indi}$ et la transitivité de l'indépendance donnent que 

$$a \ind_D \cl(D,b).$$

\noindent Comme $D_0\cup\{b,\tau(b)\}\subset \cl(D,b)$, on en déduit que 
$a \ind_{D_0} b, \tau(b)$.

Soit $\{a_i,c_i\}_{i\in\N}$ une suite de Morley du type 
$\tp(a,\tau(a)/D_0,b,\tau(b))$, qui peut être prise telle que 

$$\{a_i\}_{i\in\N}\ind_{D_0} b,\tau(b),$$

\noindent par l'indépendance précédente. Pour
$X=\alg{\sscl{D_0(b,\tau(b))}}\cap\alg{\sscl {\{a_i,c_i\}_{i\in\N}}}\cap K$,
 l'hypothèse $(\ref{H:Forking})$ entraîne que 
 $$ a,\tau(a)\ind_X D_0,b,\tau(b).$$ 

 \noindent Notons de plus que 

$$ \{a_i\}_{i\in\N}\ind_{D_0} X.$$

Ainsi, comme chaque $c_i\in\cl(D_0,\{a_i\}_{i\in\N})$, on a 
$X\subset \cl(D_0)$ par  la 
 remarque \ref{R:etranger}, donc $X\subset \cl(D_0)\cap 
\alg{\sscl{D_0(b,\tau(b))}}=D_0$ et, en particulier,

$$ a,\tau(a)\ind_{D_0} b,\tau(b).$$

Du fait que $ab$ soit également indépendant de $a$
(et de $b$, respectivement) sur $D$, par la remarque \ref{R:pairwise},
on conclut par symétrie. Le cas additif est traité de la même façon.
\end{proof}

\begin{remark}
Notons que le lemme reste vrai, sans utiliser l'hypothèse 
$(\ref{H:Forking})$, en considérant des hyperimaginaires (\emph{cf.} 
\cite[Chapter 3]{Wa00}).
\end{remark}

\section{Pas d'automorphismes bornés}

Nous disposons maintenant de tous les ingrédients nécessaires pour 
caractériser les automorphismes bornés. Cette démonstration est inspirée de 
celle de \cite[Lemma 3.11]{HO14}.

\begin{theorem}\label{T:noborne}(\cite[Théorème 15]{dL92} et
\cite[Proposition 2.9]{rK02})
 
Sur un corps $\F$ de base, soit un corps  
$(K,0,1,+,-,\cdot,F_1,\ldots,F_n)$ muni d'opérateurs satisfaisant 
les conditions suivantes~:
\begin{itemize}
\item[{\bf Hypothèse \ref{H:Residue}}] Tous les corps résiduels de la
$\F$-algèbre associée sont $\F$.
\item[{\bf Hypothèse \ref{H:Endo}}] Les endomorphismes associés
$\{\si_1,\ldots,\si_t\}$ sont surjectifs.
\item[{\bf Hypothèse \ref{H:AlgebriquementClosPlusFrob}}] Le corps $K$ est 
séparablement clos suffisamment saturé pour sa théorie $T$.
\item[{\bf Hypothèses \ref{H:types} et \ref{H:Simple}}] La théorie $T$ est 
simple et vérifie~:

\[ A\ind_C B \Longleftrightarrow \begin{minipage}{8cm}
                                  $\alg{\sscl {A\cup C} }$
et $\alg{\sscl {B\cup C} }$ sont
linéairement 
disjointes sur $\alg{\sscl C}$.
                                 \end{minipage}\]
 
Le type d'un uple $a$ sur un sous-corps $k=
\sscl k$ est déterminé par la classe de $\LL$-isomorphisme de  
$\alg{ \sscl{k(a)}} \cap K$. De plus, la clôture algébrique modèle-théorique 
d'un ensemble $A$ est egale à $\alg{\sscl A}\cap K$. 
\item[{\bf Hypothèse \ref{H:Forking}}] Si $B=\alg{\sscl B}\cap K$ et 
$\{a_i\}_{i<\omega}$ est une suite de Morley
 du type $\tp(a/B)$, alors
\[a\ind_{B\cap\alg{\sscl{\{a_i\}_{i<\omega}}}} B.\]

\end{itemize}
Alors, tout automorphisme borné  de $K$ est égal à la composition d'automorphismes associés et de leurs inverses,    ainsi que du Frobenius et de son inverse si $K$ est un corps algébriquement clos de caractéristique positive.
\end{theorem}
\begin{proof}
Soit $\tau$ un automorphisme de la structure $K$ borné au-dessus de
l'ensemble fini $D$. Puisque la trigonalisation des opérateurs donne une 
structure bi-interprétable, avec les mêmes automorphismes 
associés, sur laquelle $\tau$ induit aussi un automorphisme borné, nous 
supposerons par la suite que les opérateurs sont triangulaires, comme dans 
la proposition \ref{P:Triang_Charlotte}.

Prenons deux éléments $a$ et $b$
génériques et indépendants sur $D \cup \tau\inv(D)$. En particulier, leurs 
images $\tau(a)$ et $\tau(b)$ sont aussi génériques sur $D$.  

Posons $D_0=\cl(D)\cap \alg{\sscl{D(a,b,\tau(a),\tau(b)}}$. Le lemme
\ref{L:independanceSurCloture} entraîne que les paires $(a,\tau(a))$,
$(b,\tau(b))$ et $(a+b,\tau(a+b))$ sont deux-à-deux indépendants sur
$D_0$. De plus, les éléments $a$ et $\tau(a)$ restent génériques sur 
$D_0$. Le lemme \ref{L:stab} à l'intérieur de $\bG_a^2(K)$ donne que 
$(a,\tau(a))$ est dans un translaté, type-définissable sur $D_0$, de son 
stabilisateur additif $H\leq \bG_a^2(K)$, qui est connexe. 

Puisque $D\subset D_0$, alors $\tau(a)\in\cl(D_0,a)$, donc 

$$a \nind_{D_0} \tau(a),$$

\noindent  par la remarque \ref{R:indi}. En particulier, la paire 
$(a,\tau(a))$ n'est pas générique dans $\bG_a^2(K)$, donc l'indice de $H$ dans 
$\bG_a^2(K)$ n'est pas borné. Par le corollaire \ref{C:sousgroupesAdditifs}, 
le sous-groupe $H$ est contenu dans un sous-groupe d'indice non borné défini 
par une combinaison linéaire sur $D_0$ 
de la forme $$\lambda_{\tht_1} \tht_1( x) +S_1(x)+ \mu_{\tht_2}
\tht_2( y) + S_2(y) = 0,$$ où $S_1$, respectivement $S_2$, est une 
combinaison linéaire de mots en $x$, respectivement en $y$, à coefficients
sur $D_0$ de degré strictement inférieur à $\tht_1$, respectivement à
$\tht_2$.

Comme $(a,\tau(a))$ est dans un translaté, il existe $\xi$ dans 
$D_0$ tel que~: 

\begin{equation}\label{E:premiere} 
 \lambda_{\tht_1} \tht_1( a) + \mu_{\tht_2}
\tht_2( \tau(a)) + S_1(a) + S_2(\tau(a)) =\xi. \tag{$\blacklozenge$}
\end{equation}

\noindent De plus $\lambda_{\tht_1} \tht_1( a) +S_1(a) \neq 0 $ et
$\mu_{\tht_2}\tht_2(\tau(a)) + S_2(\tau(a)) \neq 0$, car les 
éléments $a$ et $\tau(a)$ sont chacun génériques sur $D_0$. 

Quitte à remplacer $D_0$, prenons $(\tht_1,\tht_2)$ minimal dans  
l'ordre lexicographique tel que $(a,\tau(a))$ satisfait une telle équation 
sur des paramètres au-dessus desquels $a$ et $\tau(a)$ restent génériques.

À nouveau, le lemme \ref{L:stab} à l'intérieur du groupe multiplicatif 
$\bG_m^2(K)$ donne que l'uple $(a,\tau(a))$ a un stabilisateur $H_1$ 
connexe et type-définissable sur $D_0$. Comme  $(a,\tau(a))$ est générique 
dans un translaté de $H_1$, aussi type-définissable sur $D_0$, on conclut 
que $H_1$ est infini. En particulier, il existe des éléments génériques de 
$H_1$ sur tout ensemble de paramètres. 

 Prenons $(g,h)$ un générique de $H_1$ indépendant de
$(a,\tau(a))$ sur $D_0$. Alors, l'élément $(g\cdot a, h\cdot
\tau(a))\equiv_{D_0} (a,\tau(a))$, donc 

$$\lambda_{\tht_1} \tht_1(g\cdot a) + \mu_{\tht_2}
\tht_2(h\cdot \tau(a)) + S_1(g\cdot a) + S_2(h\cdot \tau(a)) =\xi,$$

\noindent ce qui entraîne, par le corollaire \ref{C:Trig_Charlotte}

\begin{equation}\label{E:deuxieme}
 \lambda_{\tht_1}\si_{\tht_1}(g) \tht_1( a) + \mu_{\tht_2}
\si_{\tht_2}(h)\tht_2(\tau(a))
+S_1'(a) +S_2'(\tau(a)) =\xi,
\tag{$\bigstar$}
\end{equation}

\noindent pour des combinaisons linéaires $S_1'$ et
$S_2'$ sur $\sscl{D_0(g,h)}$, de degrés inférieurs à $\tht_1$ et
$\tht_2$, respectivement. Comme $a$ et $\tau(a)$ restent chacun 
générique sur $\sscl{D_0(g,h)}$, si l'on multiplie
(\ref{E:premiere}) par $\si_{\tht_2}(h)$ et l'on soustrait
(\ref{E:deuxieme}), on obtient  des combinaisons linéaires $R_1$ et
$R_2$ sur $\sscl{D_0(g,h)}$, de degrés inférieurs à $\tht_1$ et
$\tht_2$, respectivement, telles que

\[\lambda_{\tht_1} (\si_{\tht_1}(g) - \si_{\tht_2}(h)) \tht_1( a) + R_1(a) 
+ R_2(\tau(a)) =\xi(1-\si_{\tht_2}(h)).\]
\vskip2mm
\noindent  La minimalité de 
$(\tht_1,\tht_2)$ donne que 

$$\si_{\tht_1}(g) = \si_{\tht_2}(h).$$  

\noindent
Puisque $H_1$ est connexe, on en déduit que 

$$H_1\subset
\{(g,h)\in\bG_m^2\,|\, \si_{\tht_1}(g) = \si_{\tht_2}(h) \}.$$  
 Ainsi,  il existe un 
élément $\lambda_a\neq 0$ dans $D_0$ tel que $\si_{\tht_2}(\tau(a)) = \lambda_a\cdot
\si_{\tht_1}(a)$, car $(a,\tau(a))$ est dans un 
translaté multiplicatif de $H_1$ type-définissable sur $D_0$.  

Comme  $b$ est un générique quelconque indépendant de $a$ sur  $D \cup \tau\inv(D)$ et que $(b,\tau(b))$ a même stabilisateur additif $H$ que $(a,\tau(a))$, on en déduit que pour tout générique $x$ sur  $D \cup \tau\inv(D)$,
il existe un élément $\lambda_x \neq 0$ dans $K$ tel que $\si_{\tht_2}(\tau(x)) = \lambda_x\cdot
\si_{\tht_1}(x)$ (pour les mêmes mots $\tht_1$ et $\tht_2$).

De plus, on a  l'égalité~:

$$ \lambda_{a+b}
\si_{\tht_1}(a+b)=\si_{\tht_2}(\tau(a+b))=\si_{\tht_2}(\tau(a))+\si_{\tht_2}(\tau(b))=\lambda_a
\si_{\tht_1}(a) + \lambda_b
\si_{\tht_2}(b).$$

\noindent En particulier, 

$$ (\lambda_{a+b}-\lambda_a)\si_{\tht}(a)=
(\lambda_b-\lambda_{a+b})\si_{\tht}(b),$$
\noindent qui entraîne $\lambda_a=\lambda_{a+b}=\lambda_b$, car
$a\ind_{D_0} b$. 

\noindent Alors, pour tout générique $x$ sur  $D \cup \tau\inv(D)$, on a $\lambda_x = \lambda_a$ et $\lambda_a= \lambda_b= \lambda_{a \cdot b} = \lambda_{a} \cdot \lambda_b$, donc $\lambda_a =1$. Enfin, puisque tout 
élément s'écrit comme une somme de deux
 génériques, on obtient l'égalité $\si_{\tht_2}\circ \tau = \si_{\tht_1}$.

Comme le Frobenius commute avec tout automorphisme corpique, il existe deux entiers naturels $m_1$ et $m_2$, et des automorphismes $\tau_1$ et $\tau_2$ qui sont compositions de puissances entières des automorphismes associés, tels que $\si_{\tht_i} = \Frob^{m_i} \circ \tau_i$, pour $i=1, 2$. Puisque $\tau$ est un automorphisme, on a nécessairement $m_1=m_2$ si $K$ n'est pas algébriquement clos, ce qui permet de conclure.
\end{proof}

\begin{remark}\label{R:SCF} Tout corps séparablement clos de degré d'imperfection fini, 
muni d'une famille finie commutative de dérivations de Hasse-Schmidt itératives  avec une 
$p$-base canonique nommée, vérifie les hypothèses \ref{H:types}, \ref{H:Simple} et \ref{H:Forking} \cite{fD88,mZ03}. Par la remarque \ref{R:Trig_SCF}, la preuve du théorème précédent s'applique également à ce contexte. Ainsi, aucun corps séparablement clos de degré d'imperfection fini non nul  
ne possède d'automorphismes bornés non-triviaux fixant une $p$-base (rappelons 
que toute $p$-base détermine une famille finie commutative de dérivations, qui sont définissables 
au-dessus de celle-ci). 

\noindent  Nous ignorons si les hypothèses sont vérifiées pour tout corps muni de $G$-dérivations \cite{dHpK16}.
\end{remark}

Puisque les automorphismes associés à un corps muni d'opérateurs
libres existentiellement clos ne commutent pas avec les opérateurs,
ils ne sont pas des $\LL$-automorphismes, ce qui donne le résultat
suivant.

\begin{cor}\label{C:pas_de_bornes}
Un corps muni d'au moins deux opérateurs libres existentiellement
clos n'a pas d'automorphismes bornés non-triviaux. 
\end{cor}

\begin{cor}\label{C:LasEGT}(\cite[Corollaire 16]{dL92} et
\cite[Example 3.13]{EGT})
 Le groupe $\Autf(M)$ est simple, pour $M$  un corps de
caractéristique $0$ saturé dénombrable algébriquement clos ou
différentiellement clos de caractéristique $0$. Si $M$ est un corps de
caractéristique positive saturé dénombrable algébriquement clos,
alors le groupe $\Autf(M)$ est simple modulo le sous-groupe
monogène engendré par le Frobenius. 
\end{cor}

\end{document}